\documentclass[9pt]{amsart}
\usepackage{amsmath, amssymb, euscript,  enumerate}
\usepackage{color,wrapfig}
\usepackage{pxfonts}
\usepackage{verbatim} 

\newcommand{\Real}{\mathbb R}

\newcommand{\cB}{\mathcal B}

\newcommand{\cZ}{\mathcal Z}

\def\XXint#1#2#3{{\setbox0=\hbox{$#1{#2#3}{\int}$}
     \vcenter{\hbox{$#2#3$}}\kern-1\wd0}}

\newcommand{\R}{\mathbb R}

\newcommand{\cS}{\mathcal S}
\newcommand{\e}{\varepsilon}

\newcommand{\al}{\alpha}

\newcommand{\la}{\lambda}

\newcommand{\A}{|A|^2}
\newcommand{\x}{|x|^2}

\newcommand{\D}{\nabla}
\newcommand{\Di}{\nabla_{i}}
\newcommand{\Dj}{\nabla_{j}}

\newcommand{\ra}{\rightarrow}
\newcommand{\La}{\triangle}

\newtheorem{thm}{Theorem}[section]
\newtheorem{lem}[thm]{Lemma}
\newtheorem{condition}[thm]{Condition}

\newtheorem{definition}[thm]{Definition}
\newtheorem{corollary}[thm]{Corollary}

\theoremstyle{remark}
\newtheorem{remark}[thm]{Remark}
\numberwithin{equation}{section}

\begin{document}
{\Large
\title{$\alpha$-Gauss Curvature flows with flat sides}
\author{ Lami Kim}
\address{Seoul National University, Seoul 151-747, Korea }
\email{lmkim@snu.ac.kr}
\author{Ki-ahm Lee}
\address{Seoul National University, Seoul 151-747, Korea}
\email{kiahm@snu.ac.kr}
\author{ Eunjai Rhee}
\address{Seoul National University, Seoul 151-747, Korea }
\email{erhee@math4.snu.ac.kr}

\maketitle

\begin{abstract}
In this paper, we study the deformation of the 2 dimensional convex surfaces in $\R^{3}$ whose speed at a point on the surface is proportional to $\alpha$-power of positive part of Gauss Curvature.
First, for $\frac{1}{2}<\alpha\leq 1$, we show that there is smooth solution if the initial data is smooth and strictly convex and that there is a viscosity solution with $C^{1,1}$-estimate before the collapsing time if the initial surface is only convex. Moreover, we show that there is a waiting time effect  which means the flat spot of the convex surface will persist for a while.
We also show the interface between the flat side and the strictly convex side of the surface remains smooth on $0 < t < T_0$ under certain necessary regularity and non-degeneracy initial conditions, where $ T_0$ is the vanishing time of the flat side.
 \end{abstract}

\section{Introduction}

We are concerned with the regularity of the $\alpha$-Gauss Curvature flow with flat sides, which is associated to the free boundary problem. This flow explains the deformation of a compact convex subjects moving with collision from any random angle. The probability of impact at any point $P$ on the surface $\Sigma$  is proportional to the $\alpha$-Gauss Curvature $K^{\alpha}$. Then the deformation of the surface $\Sigma$ can be described by the flow
\begin{equation}\label{eqn-1.0}
\begin{split}
\frac {\partial X}{\partial t} (x,t)&=- K^{\alpha}(x,t)\, \nu (x,t)\\
X(x,0)&=X_0(x)
\end{split}
\end{equation}
where $\nu$ denotes the unit outward normal and $\alpha>0$.
Now we are going to summarize the known results for the evolution of
the {\it strictly convex surfaces} following \eqref{eqn-1.0}, \cite{C1, C2}.

Various application of \eqref{eqn-1.0} has been discussed at
\cite{A2}: rolling stone on the hyperplanes ($\alpha=1$, \cite{F}),
the affine normal flows ($\alpha=\frac{1}{n+2}$, \cite{ST1,ST2}),
the gradient flows of the mean width in $L^p$-norm
($\alpha=\frac{1}{p-1}$, \cite{A2}), and  image process ($p=\frac14$,
\cite{AGLM}).

The dynamics and degeneracy of the diffusion varies depending on
$\alpha$. If $\alpha$ is smaller, it becomes more singular and the
solution gets regular instantaneously. On the other hand, if $\alpha$
is greater than $\frac{1}{n}$, it becomes degenerate and has waiting
time effect which means that the flat spot of the surface stays for a
while, \cite{A2}. Waiting time and finite speed of propagation caused by the degeneracy have been studied at the
 other well-known degenerate equations:
the Porous Medium Equation
$$u_t=\La u^m$$
and Parabolic $p$-Laplace Equation
$$u_t=\D\cdot (|\D u|^{p-2}\D u).$$

\subsection{The known results}\label{sec-intro-history}
Let $X(\cdot,\cdot):S^n\times [0,T)
\ra\R^{n+1}$ be a embedding and set $\Sigma_t=X(S^n,t)$. Since the
volume is decreasing in time and vanishes at finite time, there is  the first time, $T_0$, when
$vol(\Sigma_t)$ becomes zero.

Let us assume $\Sigma_0$ be strictly convex and
  smooth. Then $\Sigma_t$  is also smooth and strictly
  convex for $0<t<T_0$ for $\alpha=1$, \cite{T}.
If we rescale  $\Sigma_t$ to $\tilde{\Sigma}_t$ so that
  $vol(\tilde{\Sigma}_t )=vol(\Sigma_0)$, then   $\tilde{\Sigma}_t$  converges to a
  sphere, for $\alpha=1$ and $n=2$ \cite{A1} or for $\alpha=\frac{1}{n}$
  and $n\geq 2$, \cite{C1}. It is also true if $\alpha>\frac{1}{n}$ and
  if the initial surface is sufficiently close to the sphere, \cite{C1}.
And $\Sigma_t$ converges to a point if  $\alpha\in
  (\frac{1}{n+2},\frac{1}{n}]$, or if $\alpha\in (0,\frac{1}{n}]$ and $\Sigma_t$ has
  bounded isoperimetric ratio which is the ratio between the radius of
  the inner sphere and that of outer sphere, \cite{A2}.

Now let  $\Sigma_0$ be convex and
  smooth. For $\alpha>0$, there is a viscosity solution, $\Sigma_t$,  for $0<t<T_0$ which has
  uniform Lipschitz bound, \cite{A2}.
For $\frac{1}{2}<\alpha\leq 1$ and $n=2$, the convex viscosity solution, $\Sigma_t$, has a uniform
  $C^{1,1}$-estimate  for $0<t<T_0$, \cite{KL}.
For $\alpha = 1$ and $n=2$, the $C^{\infty}_{\delta}$-regularity of the strictly convex part of  the surface and the
  smoothness of the interface between the strictly convex part and
  flat spot have been proved at \cite{DL3}.

\subsection{The balance of terms}
In this paper, we are going to study  the regularity of $\Sigma_t$, when
the initial surface, $\Sigma_0$, has a flat spot for $n=2$.

We will assume for simplicity that the initial surface $\Sigma_0$ has
only one flat spot, namely that at $t$ we have
$\Sigma_t=\Sigma^1_t\cup\Sigma^2_t$ where $\Sigma^1_t$ is the flat spot and
$\Sigma^2_t$ is strictly convex part of $\Sigma_t$. The intersection between two regions is the free boundary $\Gamma_t=\Sigma^1_t\cap\Sigma^2_t$. The lower part of the surface $\Sigma_0$ can be written as a graph $z=f(x)$. And similarly we can write the lower part of $\Sigma_t$ as $z=f(x,t)$ for $x\in\Omega\subset\R^n$ where $\Omega$ is an open subset of $\R^n$.

The function $f(x,t)$ satisfies $\alpha$-Gauss Curvature flow:
\begin{equation}\label{eqn-f}
f_{t}=\frac{[\det (D^{2}f)]^{\alpha}}{(1+|\D
f|^2)^{\frac{\alpha(n+2)-1}{2}}}.
\end{equation}

Let's consider rotationally symmetric case first to see the balance between terms for $n=2$. If $f=f(r)$ is rotationally symmetric, $\eqref{eqn-f}$ can be written as
\begin{equation}\label{eqn-f1}
\begin{split}
f_t=\frac{f_r^{\alpha}f_{rr}^{\alpha}}{r^{\alpha}(1+f_r^2)^\frac{4\alpha-1}{2}}
\end{split}
\end{equation}

Let $r=\gamma (t)$ be the equation of the free boundary $\Gamma(f)=\partial\{f=0\}$.
The speed of boundary is given by
$$\gamma_t=-\frac{f_t}{f_r}=-\frac{f_r^{\alpha-1}f_{rr}^{\alpha}}{r^{\alpha}(1+f_r^2)^\frac{4\alpha-1}{2}}.$$

The regularity comes from the nondegenerate finite speed of the
free boundary before the flat spot converges to a lower
dimensional singularity at a focusing time. When
$f=(r-1)_+^{\beta}$ at a given time $t$, for $r \approx 1$,
$$ |\gamma_t|= s^{(\alpha-1)(\beta-1)} s^{\alpha (\beta-2)}\approx 1$$
for $s=r-1$, which implies $\beta=\frac{3\alpha-1}{2\alpha-1}$.

For the general $f=f(x,y,t)$, let $f=\frac{1}{\beta}g^{\beta}$ for
$\beta=\frac{3\alpha-1}{2\alpha-1}$. The equation for this pressure $g$ will be
\begin{equation}\label{eqn-g}
g_{t}=\frac{[g\det (D^{2}g)
+\theta(\alpha)(g_{x}^{2}g_{yy}+g_{y}^{2}g_{xx}-2g_{x}g_{y}g_{xy})]^{\alpha}}{(1+g^{2\beta-2}|\D
g|^2)^{\frac{4\alpha-1}{2}}}
\end{equation}
for $\theta(\alpha)=\beta-1=\frac{\alpha}{2\alpha-1}$.

Assuming $g_{\tau}=0$ at the boundary, the speed of boundary will
be
\begin{equation}
\gamma_{t}=-\frac{g_{t}}{g_{\nu}}=-\theta(\alpha)^{\al} \, g_{\nu}^{2\alpha-1} g_{\tau\tau}^{\alpha}
\end{equation}
for a tangential direction $\tau$ and a normal direction $\nu$ to $\partial \Omega$.

\subsection{Conditions for $f$}
\begin{condition}\label{cond}
Set $\Lambda (f)=\{f=0\}$ and $\Gamma (f)=\partial\Lambda (f)$.

\begin{enumerate}[(I)]
\item{({\it Nondegeneracy Condition})}
Our basic assumption on the initial surface is that the function
$f$ vanishes of the order $\text{dist}
(X,\Lambda(f))^{\frac{3\alpha-1}{2\alpha-1}}$ and that the
interface $\Gamma(f)$ is strictly convex so that the interface
moves with finite nondegenerate speed. Namely, setting $g=(\beta
f)^{\frac{1}{\beta}}$, we assume that at time $t=0$ the function
$g$ satisfies the following nondegeneracy condition: at $t=0$,
\begin{equation}\label{eqn-condition} 0<\lambda <| Dg(X) | < \frac{1}{\lambda} \quad
\quad \text{and} \quad 0<\lambda^2 <D_{\tau\tau}^2 g(X) <
\frac{1}{\lambda^2 } 
\end{equation}
for all $ X  \in  \Gamma _0$ and some positive number $\lambda >0$, where
$D^2_{\tau\tau}$ denotes the second order tangential
derivative at $\Gamma$.
Then  the initial speed of free boundary has the speed, at $t=0$,
\begin{equation}
0<\lambda^{4\alpha-1} <| \gamma_{t}| <
\frac{1}{\lambda^{4\alpha-1}}.
\end{equation}
\item{({\it Before Focusing of Flat Spot})}
Let $T$ be any number on $0<T<T_0$, so that the flat side
$\Sigma^1_t$ is non-zero. Since the area is non-zero,
$\Sigma^1_t$ contains a disc $D_{\rho_0}$ for some $\rho_0>0$. We
may assume that
\begin{equation}\label{co1}
D_{\rho_0} = \left\{ X \in R^2\,:\,|X| \le
\rho_0\right\} \subset \Sigma^1_t\quad\text{for $0\leq t\leq T_0$}.
\end{equation}
\item{({\it Graph on a Neighborhood of the Flat spot $\Sigma^1_t$})} We will also assume, without loss of generality, throughout the paper that
\begin{equation}\label{maxf}
\max_{x\in \Omega (t)} f(\cdot ,t)\geq 2 , \quad 0\leq t\leq T_0 
\end{equation}
where $\Omega (t)=\{X=(x,y)\in R^2 \,:\,|D f|(X,t) < \infty \}$.
Set
\begin{equation}\label{omeg}
\Omega_P(t)= \left\{ x\in R^2\,:\,f(x,y,t) \le f(P)\right\}.
\end{equation}

\end{enumerate}

\end{condition}

\subsection{The concept of regularity} Let's assume $P_0=(x_0,y_0,t_0)$ is an interface point and $t_0$ is sufficiently small. Then condition $\eqref{eqn-condition}$ is satisfied at $t_0$ for small constant $c$. We can assume 
\begin{equation}\label{co2}
g_x(P_0) \geq c >0 \quad \text{for some} \quad c >0
\end{equation}
by rotating the coordinates. Also by transforming the free-boundary to a fixed boundary near $P_0$, we can obtain the map $x=h(z,y,t)$ where $(z,y,t)$ is around $Q_0=(0,y_0,t_0)$ and then the free-boundary $g=0 $ is transformed into the fixed boundary $z=0$. From the calculation on $g(h(z,y,t),y,t)=z$, we have the fully nonlinear degenerate equation
\begin{equation}\label{eqn-hhhh}
-\frac{h_t}{h_{z}} = \frac{(z\, ( h_{zz} h_{yy} -h_{zy}^2) -
\theta(\al)h_z
h_{yy})^{\alpha}}{h_{z}^{4\alpha}}\frac{h_{z}^{4\alpha-1}} {(z^2 +
h_z^2+ z^2 h_y^2)^{\frac{4\alpha-1}{2}}}, \qquad z >0
\end{equation}
implying that under \eqref{eqn-condition} and initial regularity conditions, the linearized operator
\begin{equation}\label{eqn-hl}
\tilde  h_t  = z\, a_{11} \tilde h_{zz} + 2 \sqrt z \, a_{12} \tilde h_{zy}
+ a_{22} \tilde h_{yy} + b_1 \tilde h_z + b_2 \tilde h_y 
\end{equation}
where $(a_{ij})$ is strictly positive and $b_1 \geq \nu >0$ for some $\nu >0$.

\begin{definition} \label{def-sp} For the Riemannian metric $ds$ with $ds^2 = {dz^2 \over z}+dy^2\ $, let distance between $Q_1=(z_1,y_1)$ and $Q_2=(z_2,y_2)$
in the metric $s$ be $s(Q_1,Q_2) = |\sqrt {z_1} - \sqrt{z_2}| + |y_1-y_2|$ and the parabolic distance between $ Q_1=(z_1,y_1,t_1)$ and $Q_2=(z_2,y_2,t_2)$ be $s(Q_1,Q_2) = |\sqrt {z_1} - \sqrt{z_2}| + |y_1-y_2| + \sqrt{|t_1-t_2|}$. Then we define $C^{\gamma}_s$, $\gamma \in (0,1)$  as the space  of
 H\"older continuous functions with respect to the metric $s$ and $C^{2+\gamma}_s$ as the space of all functions  $h$ with 
$$h, h_z, h_y, h_t,   z\, h_{zz}, z\sqrt z\,  h_{zy}, h_{yy} \in C^\gamma_s.$$
\end{definition}

\begin{remark}
When we consider the equation 
\begin{equation} \label{eqn-mod}
h_t = z\,  h_{zz} + h_{yy}  + \nu \,h_z
\end{equation}
on the half-space with $\nu >0$, which doesn't have the other condition of h on $z=0$, the Riemannian metric $ds$ decides the diffusion of the equation.
\end{remark}

\begin{remark}
If the transformed function $h \in C^{2+\gamma}_s$, we say that $g \in C^{2+\gamma}_s$ around the interface $\Gamma$.
\end{remark}
\subsection{Main Theorems}
Now we are going to state main theorems when the initial surface is just convex having a flat spot.

\begin{thm}\label{thm-1}
Let us assume $\frac12<\alpha\leq 1$. If $\Sigma_0$ is  convex, any viscosity solution  $\Sigma_t$ of \eqref{eqn-1.0} is $C^{1,1}$ for $0<t<T_0$. Moreover the strictly convex part,  $\Sigma^2_t$,  is smooth for $0<t<T_0$.
\end{thm}

The following short time existence of  $C^{\infty}_s$-solution with a flat spot has be essentially proved in
\cite{DH} since the linearized  equation for $h$, \eqref{eqn-hy}, is in the same class of
operators considered in \cite{DH} because of the
conditions, \eqref{eqn-condition}, as \cite{DH}.
Therefore  their Schauder theory can be applied to \eqref{eqn-hy} and then
the application of implicit function theorem gives the short time
existence as  \cite{DH}.
\begin{thm}\label{thm-1.3}[Short Time Regularity]\cite{DH}
For $\frac12<\alpha\leq 1$, assume that $g=(\beta f)^{\frac{1}{\beta}}$
is of class $C^{2+\gamma}$ up to the interface $z=0$ at time $t=0$, for
some $0<\gamma<1$, and satisfies Conditions \ref{cond} for $f$. Then there exists a time $T>0$ such that the
$\alpha-$Gauss Curvature Flow \eqref{eqn-1.0}  admits a solution $\Sigma(t)$ on $0\leq t \leq T$. In addition the function $g=(\beta f)^{\frac{1}{\beta}}$ is smooth up to the interface $z=0$ on $0< t \leq T$. In particular the junction $\Gamma (t)$ between the strictly convex and the flat side will be a smooth curve for all $t$ in $0< t \leq T$.
\end{thm}
One of  main results in this paper is the following long time regularity
of the solution.
\begin{thm}\label{thm-1.4}[Long Time Regularity]
Under the assumptions of Theorem \ref{thm-1.3}, the function $g=(\beta
f)^{\frac{1}{\beta}}$ remains smooth up to the interface $z=0$ on $0< t
<T$ for all $T< T_0$. And the interface  $\Gamma_t$ between the
strictly convex and the flat side will be smooth curve for all $t$ in $0<t<T_0$.
\end{thm}

To show Theorem \ref{thm-1.4},  we follow the main steps at \cite{DL3}. But the
 exponent $\alpha$ creates large number of nontrivial terms especially in the estimate
 of the second derivatives. New quantities have been considered to
 absorb the effect of  terms depending on $(1-\alpha)$ at Lemma \ref{lem-43}. Optimal regularity  and  Aronson-B\'{e}nilan type estimate have been proved at Lemma \ref{lem-AB} and \ref{lem-43-1}.

\section{Convex surface.}
\subsection{Evolution of the metric and curvature}

The metric and second fundamental form can be defined by
$$g_{ij}=\left<\frac{\partial X}{\partial x^i},\frac{\partial X}{\partial x^i}\right> \quad \text{and } \quad  h_{ij}=-\left<  \frac{\partial^2 X}{\partial x^i\partial x^j},\nu\right>$$
with respect to a local coordinates $\{x_{1},\cdots,x_{n}\}$ of $\Sigma_{t}$ and $\nu$ is the outward unit normal to $\Sigma_{t}$. Also the Weingarten map is given by
$$h^{i}_{j}=g^{ik}h_{kj},$$
and then $\sigma_{k}=\sum_{1\leq i_{1} < \cdots < i_{k}\leq n}\lambda_{i_{1}}\lambda_{i_{2}}\cdots\la_{i_{k}}$,
$H=\text{trace}(h)=\sigma_{1}=\sum_{1\leq i\leq n}\lambda_i$,
$K=\det (h)=\sigma_{n}=\lambda_{1}\lambda_{2}\cdots\la_{n}$, and
$|A|^{2}=h_{ij}h^{ij}=\lambda_{1}^{2}+\cdots +\lambda_{n}^{2}$ where $\lambda_{1},\cdots,\lambda_{n}$ are the eigenvalues of the Weingarten map.

The evolution of the metric, second fundamental form, and curvature are the following.
\begin{lem}\label{ev-h}
Let $X(x,t)$ be a smooth solution of \eqref{eqn-1.0}. Then we have the following. The proof can be referred to Chapter 2, \cite{Z}. Let $\Box$ denote $K^{\al} (h^{-1})^{kl} \D _k \D _l$.

    \begin{enumerate}[(i)]
\item $\displaystyle\frac{\partial g_{ij}}{\partial t}=-2K^{\alpha}h_{ij}$
\item $\displaystyle\frac{\partial N}{\partial t}=-g^{ij}\frac{\partial K^{\al}}{\partial x^i}\frac{\partial X}{\partial x^j}=-\D^j K^{\al}\frac{\partial X}{\partial x^j} $
\item $  \displaystyle\frac{\partial h_{ij}}{\partial t}=\al \Box h_{ij}+\al ^2 K^{\al}(h^{-1})^{kl}(h^{-1})^{mn}\D_ih_{kl}\D_jh_{mn}-\al K^{\al}(h^{-1})^{km}(h^{-1})^{nl}\D_ih_{mn}\D_jh_{kl}$\\
           $\quad +\al K^{\al}H h_{ij}-(1+2\al)K^{\al}h_{jl}h^{l} _{i}$ 
\item $\displaystyle\frac{\partial K}{\partial t}=\al \Box K+\al(\al-1)K^{\al -1}(h^{-1})^{ij}\D_i K\D_j K+K^{\al +1}H$
\item $\displaystyle\frac{\partial K^{\al}}{\partial t}=\al \Box K^{\al}+\al K^{2\al}H$
\item $ \displaystyle  \frac{\partial H}{\partial t} =\al \Box H+{\al}^2 {K}^{\al -2}g^{ij}\D_iK\D_jK-\al K^{\al}g^{ij}(h^{-1})^{km}(h^{-1})^{nl}\D_ih_{mn}\D_jh_{kl} +\al K^{\al }H^2 $\\
$+(1-2\al )K^{\al}|A|^2 .$

\end{enumerate}
\end{lem}

\subsection{Curvature Estimates}

Now we are going to show the regularity of $\Sigma_t$. The following Lemma was proved in \cite{A2, KL}.
\begin{lem} \label{lem-uK} \cite {A2, KL}  Let $\Sigma_0$ be convex and $\alpha>0$. Then
\item 
\begin{enumerate}[(i)]
\item  There is a constant $C>0$ such that
$$ \sup_{x\in \Sigma ,\ 0<t<T_0 } K^{\al} (x,t)\leq C(\alpha)=\max \Bigg(\sup_{x\in \Sigma } K^{\al}(x,0), \bigg( \displaystyle\frac{2\al +1}{2 \al \rho_0} \bigg) ^{2\al }\Bigg) .$$

\item  $\inf_{x\in \Sigma ,\ 0<t<T_0 }K^{\al} \geq  e^{C(t_0)t}\inf_{x\in \Sigma}K^{\al}(x,0) $ where $C(t_0)$ is some constant for $0<t_0<T_0 $ 
\item There is a unique viscosity solution $\Sigma_t$.
\end{enumerate}
\end{lem}

\begin{lem} \label{lem-uH}
Set $\psi (x,t)= <x, \nu >$ and let $B_{R_0}(0)$ be a ball of radius $R_0$ about the origin and $ P = \displaystyle \frac{H}{\psi+4R^2-\x }$, where $\Sigma_{0}$ is contained in $B_{R_0}(0)$ and $R^2=\max({R_0}^2, R_0)$. 
Then there exists a constant $C=C(\sup_{x\in \Sigma ,\ 0\leq t< T_0 } K^{\alpha}, R)>0$ for $\frac{1}{2} < \alpha \leq 1 $ such that
\begin{align*}
   \sup_{x\in \Sigma ,\ 0\leq t< T_0 } H(x,t)\leq C .
\end{align*}
\end{lem}

\begin{proof}
Since $|x|$ is decreasing, $\psi+4R^2-\x $ is positive and then we have 
\begin{equation*}
\begin{split}
\frac{\partial}{\partial t} \x &=\Box \x + 2K^{\al} \langle x,\nu\rangle - 2K^{\al }(h^{-1})^{kl}g_{kl}.
\end{split}
\end{equation*}
By using $\Di P =0$ at the maximum point, we can obtain
\begin{align*}
\al \Box P = \frac{\al \Box H }{\psi+4R^2-\x} + \frac{\al H\Box \x}{(\psi+4R^2-\x )^2}-\frac{\al H \Box \psi }{(\psi+4R^2-\x )^2}
\end{align*}
and then since $\Di\Dj P \leq 0$ at the maximum point, we get
\begin{equation}\label{third}
\begin{split}
\frac{\partial}{\partial t} P &\leq \frac{(1-\al) H\Box \x}{(\psi+4R^2-\x )^2} +\frac{H \big( (2\al +1)K^{\al} + 2K^{\al } \psi \big)  }{( \psi+4R^2-\x )^2} - \frac{ H^2 (\al K^{\al} \psi +2K^{\al -1}) }{(\psi+4R^2-\x )^2} \\
&\quad + \frac{\al K ^{\al}}{\psi+4R^2-\x}\big(\al g^{ij}(h^{-1})^{kl}(h^{-1})^{mn}\D_ih_{kl}\D_jh_{mn} - g^{ij}(h^{-1})^{km}(h^{-1})^{nl}\D_ih_{mn}\D_jh_{kl} \big) \\
& \quad + \frac{1}{\psi+4R^2-\x }\big(\al K ^{\al } H^2 + (1-2\al )K^{\al } \A  \big)
\end{split}
\end{equation}
at the maximum point.
Now, we can estimate the third term of (\ref{third}) by the following inequality
\begin{align*} 
&\al g^{ij}(h^{-1})^{kl}(h^{-1})^{mn}\D_ih_{kl}\D_jh_{mn} - g^{ij}(h^{-1})^{km}(h^{-1})^{nl}\D_ih_{mn}\D_jh_{kl}\\
&\quad = (\al -1)\Bigg{[} \big((h^{-1})^{11}\D_1 h_{11} + (h^{-1})^{22}\D_1 h_{22}\big)^2 + \big((h^{-1})^{11}\D_2 h_{11} + (h^{-1})^{22}\D_2 h_{22}\big)^2   \Bigg{]} \\
&\quad \quad +2 (h^{-1})^{11}(h^{-1})^{22}\big{\{} \D_1 h_{11} \D_1 h_{22} +\D_2 h_{11} \D_2 h_{22} \big{\}}  -2 (h^{-1})^{11}(h^{-1})^{22}\big{\{} (\D_2 h_{11})^2 +(\D_1 h_{22})^2 \big{\}} \\
&\quad \leq 2 (h^{-1})^{11}(h^{-1})^{22}\big{\{} -(\D_1 h_{22} )^2 -P\D_1 h_{22}(\D_1 \x -\D_1 \psi )-(\D_2 h_{11})^2 -P\D_2 h_{11}(\D_2 \x -\D_2 \psi) \big{\}}\\
&\quad \leq  2 (h^{-1})^{11}(h^{-1})^{22}\Big(1-\frac{\tilde{h}}{2}\Big)^2 \Big(\x -\langle x ,\nu\rangle ^2 \Big) P^2
\end{align*}

where $\tilde{h} = \min{\{} {h_{11}, h_{22}}{\}}$. 
So
\begin{equation*}
\begin{split}
\frac{\partial}{\partial t} P \leq &\Bigg{[} \frac{2\al K ^{\al -1}\Big(1-\frac{\tilde{h}}{2}\Big)^2} {\psi+4R^2-\x}(\x -\psi ^2 ) + K ^{\al }\big{\{}4(1-\al)R^2-(1-\al)\x\big{\}}+(1-2\al)K ^{\al }\psi -2K^{\al -1}  \Bigg{]}P^2 \\
&+ \frac{1}{\psi+4R^2-\x }\Bigg[ (1-\al) \Box \x +\big{\{}(2\psi +(2\al+1)\big{\}}K^{\al } \Bigg{]}P + \frac{2(2\al-1) K^{\al+1 }}{\psi+4R^2-\x} .
\end{split}
\end{equation*}
For $\displaystyle\frac{1}{2} < \al \leq 1 $, we can make the coefficient of $P^2$ be negative, which can be achieved if we consider $\eta$ small enough.  The reason is if we begin with $\eta \Sigma_0$ for any given $\Sigma_0$, we can make $K \geq \displaystyle \frac{C_0}{\eta ^2}$ where $C_0$ is some constant
 depending on initial surface, which comes from Lemma \ref{lem-uK}, and $|x|^2 \leq \eta ^2$, $R^2 \leq \eta ^2$, and $\psi \leq \eta $ for sufficiently small $\eta$. Then the first term and second term of coefficient of $P^2$ are $O(\eta ^{2- 2\al})$ and the third term is negative with $K ^{\al }\psi = O(\eta ^{1- 2\al})$ for $\eta $ small enough. This implies $\displaystyle\frac{\partial P }{\partial t} \leq - \frac{1}{2} P^2 + C$ where $C=C(\sup_{x\in \Sigma ,\ 0\leq t< T_0 } K^{\alpha}, R)$
and then if $ \displaystyle -\frac{1}{2} P^2 + C < 0 $, it is contradiction. So 
$P$ is bounded and hence H is bounded before $\Sigma$ shrinks a point.
\end{proof}

\subsection{Strictly convexity away from the flat spot}
To apply Harnack principle, let us introduce new coordinate defined on the sphere $S^n$. If $\Sigma_t$ is strictly convex, $\nu(x,t)$ is a one-to-one map from $\Sigma_t$ to $S^n$, which means for each $z\in S^n$, there is $X(x,t)=\nu^{-1}(z,t)$. $K(z,t)$ denotes Gauss Curvature $K$ at $\nu^{-1}(z,t)$. If  $\Sigma_t$ is  convex, we still use the same coordinate $(z,t)$ for strictly convex part $\Sigma_t^2$ by using approximates with strictly convex surfaces.
\begin{lem}  \label{lem-lK} Assume that the flat spot, $\Sigma_0^1$, is a part of the plane orthogonal to $e_{n+1}$. For any $\eta>0$, there is a constant $c_{\eta}>0$ such that 
\begin{align*}
K^{\al}(z,t) \geq c_{\eta}
\end{align*}
for $z\in \cS_{\eta}:=\{z\in S^{n}\,\text{and}\,\|z+e_{n+1}\|\geq \eta>0\}$.
\end{lem}
\begin{proof}
 We can immediately obtain the result from the Harnack estimate in \cite{C3}:\\
 For any points $z_1,z_2\in\cS_{\eta}$ and times $0\leq t_1<t_2$
         $$\frac{K^{\alpha}(z_2,t_2)}{K^{\alpha}(z_1,t_1)}\geq
         e^{-\Theta/4}\left(
           \frac{t_2}{t_1}\right)^{-(1+(2\alpha)^{-1})^{-1}}$$
         where
         $\Theta=\Theta(z_1,z_2,t_1,t_2)=\inf_{\gamma}\int_{t_1}^{t_2}|d_t\gamma
         (t)|^2_{m(t)}dt$
         and the infimum is taken over all paths $\gamma$ in $\Sigma$
         whose graph $(\gamma(t),t)$ joins $(z_1,t_1)$ to $(z_2,t_2)$. The short time existence of smooth surfaces implies that,  for  $z\in \cS_{\eta}$, $X(z,t)$ is the strictly convex part, $\Sigma_t^2$, for $0\leq t\leq\delta_0$ for some $\delta_0>0$. Therefore we can take $0<\delta_0\leq t_1<t_2\leq T$, which implies $K^{\alpha}(z_2,t_2)\geq c_1K^{\alpha}(z_1,\delta_0)\geq c_{\eta}$ for some $c_1, c_{\eta}>0$ and then the conclusion.
\end{proof}

We finally know \eqref{eqn-f} is uniformly parabolic, which comes from Lemmas \ref{lem-uK}-\ref{lem-uH}. And then we can show that 
$\Sigma _t$ is  $C^{\infty}$ on the point being away from flat spot.
\begin{corollary} Under the same condition of Lemma \ref{lem-lK},
$\tilde{\Sigma}_t^2:=\{X(z,t)\in \Sigma_t^2: z\in\cS_{2\eta}\}$ is smooth.
\end{corollary}
\begin{proof}
Let $\lambda_i$ be the eigenvalues of $(h^i_j)$. From the convexity,   $\lambda_i\geq 0$. And from the upper bound of Mean
Curvature and the lower bound of Gauss Curvature,  $\lambda_1+\cdots+\lambda_n<C_1$ and $K=\lambda_1\cdots\lambda_n>c_2$. Now we have
$$C_1 \geq \lambda_i\geq \frac{c_2}{\Pi_{j\neq i}\lambda_j}\geq \frac{c_2}{C_1^{n-1}}>0.$$
It implies there are $0<\lambda\leq\Lambda<\infty$ such that
$$\lambda |\xi|^2\leq K(h^{-1})^{ij}\xi_i\xi_j\leq\Lambda |\xi|^2$$
and the support function $S(z,t)$ satisfies a uniformly parabolic equation in $\tilde{\Sigma}_t^2$. Therefore  $S(z,t)$ is $C^{2,\gamma}$ and then $C^{\infty}$ in $\tilde{\Sigma}_t^2$ through the standard bootstrap argument using the Schauder theory.
\end{proof}

\subsection{Proof of Theorem \ref{thm-1}}
Recall that $|A|^2$ is the square sum of principle curvatures of a given surface.
First, we approximate the initial surface $\Sigma_{0}$ with strictly convex smooth functions, $\Sigma_{0,\e}$ whose $|A_{0,\e}|^2$ is uniformly bounded by $2|A_0|^2$ of $\Sigma_{0}$. Then  there are smooth solutions  $\Sigma_{t,\e}$  of \eqref{eqn-1.0}, \cite{KL}, and $|A_{0,\e}|^2\leq 2H_{\e}^2<4|A_0|^2<C$ uniformly. As $\e\ra 0$, $\Sigma_{t,\e}$ converges to a viscosity solution $\Sigma_{t}$ as \cite{A1}.
$|A_t|^2$ of $\Sigma_{t}$ will be uniformly bounded, which implies that $\Sigma_{t}$ is $C^{1,1}$. And for any $X\in \Sigma^2_{t}$, there is a small $\eta>0$ such that $\|\nu_{X}+e_{n+1}\|\geq \eta>0$ and then $X\in  \tilde{\Sigma}^2_{t}$.
Since $ \tilde{\Sigma}^2_{t}$ is smooth at $X$, so is $\Sigma^2_{t}$.

\subsection{A Waiting Time Effect}
We now are going to show the flat spot of the convex surface will persist for some time.
\begin{lem}\label{lem-waiting-time} Let $\Sigma_0$ be convex. For $\frac{1}{2}<\alpha \leq 1$, there is a waiting time of flat spot: if $P_0\in int_n(\Sigma_0\cap\Pi)$ where $\Pi$ is a $n$-dimensional plane and $int_n(A)$ is the interior of $A$ with respect to the topology in $\Pi$, there is $t_0>0$ such that
 $P_0\in int_n(\Sigma_t\cap\Pi)$ for $0<t<t_0$.
\end{lem}

\begin{proof}
Let $h^+=C_+\frac{|X-P_0|^{\mu}}{(T-t)^{\gamma}}$ for
$\mu=\frac{4\alpha }{2\alpha-1}$, $\gamma=\frac{1}{2\alpha-1}$, and
$C_+=\Big(\frac{\gamma}{\mu^{2\alpha}(\mu-1)^{\alpha}}\Big)^{\frac{1}{2\alpha
    -1}}$. Then $h^+$ is a super-solution of \eqref{eqn-f}.
Now we are going to compare the solution $f$ with $h^+$. From
$C^{1,1}$-estimate of $f$, $f_t$ is bounded and then there is a  ball
$B_{\rho_0}(P_0)\subset int_n(\Sigma_0\cap\Pi)$ and $t_0>0$ such that
$ f(X,t)\leq h^+(X,t)$ on $\partial B_{\rho_0}(P_0)$ for $0\leq t\leq
t_0$ and $f(X,0)\leq h^+(X,0)$. From the comparison principle, we have  $ f(X,t)\leq h^+(X,t)$
for $(X,t)\in B_{\rho_0}(P_0)\times [0,t_0)$, which implies $
f(P_0,t)=0$ and $P_0\in  \Sigma_t$ for $0\leq t\leq t_0$. 
\end{proof}

\section{Optimal Gradient Estimate  near Free Interface}

\subsection{Finite and Non-Degenerate Speed of level sets}\label{sec_3.1}
From using the differential Harnack inequalities, we can show that the free-boundary $\Gamma(t)$
has finite and non-degenerate speed as \cite {DL3}. 
As Theorem \ref{thm-1.3}, we assume that $z=f(x,t)$ is a solution of \eqref{eqn-f} and $C^{1,1}$ on
$\varOmega(t)$ for all $0 < t \leq T$ and $g=(\beta f)^\frac{1}{\beta}$ is smooth up to the interface $\Gamma(t)$ on $0 < t \leq \tau $ for some $\tau < T$. 

Let us consider the function
\begin{equation}\label{eqn-2.1.1}
\begin{split}
f_{\epsilon}(x,t)=\frac{(1 -
A\epsilon)^{(4\alpha-1)/2}\,(1+\epsilon)^{4\al}}{ (1+B
\epsilon)^{2\al-1}}\, f((1+\epsilon)x,(1-A\epsilon)t)
\end{split}
\end{equation}
and then the consequences of \cite {DL3} can be applied to our equation by the similar ways.

We may assume condition \eqref{co1} and let $r=\gamma(\theta,t)$ be the interface $\Gamma(t)$ and $r=\gamma_\e(\theta,t)$ be the $\e$-level set of the function $f$ with $0 \leq \theta < 2\pi$ by expressing in polar coordinates. Then 

\begin{lem}\label{thm-24} 
There exist constants $A,B,C>0$ and $\tilde{A},\tilde{B},\tilde{C}>0$ such that
\begin{equation} \label{eqn-fb1}
e^{-\frac{t-t_0}{B+AT}} \, \gamma(\theta,t_0)\geq \gamma(\theta,t)\geq e^{-\frac{t-t_0}{Ct_0}}\gamma(\theta,t_0)
\end{equation}
and 
\begin{equation} \label{eqn-fb3}
e^{-\frac{t-t_0}{\tilde{B}+\tilde{A}T}} \, \gamma_{\e}(\theta,t_0)\geq  \gamma_\e(\theta,t)\geq e^{-\frac{t-t_0}{\tilde{C} t_0}}\gamma_\e(\theta,t_0)
\end{equation} 
for all $0< t_0\leq t\leq T$, $0 \leq \theta < 2\pi$. In
particular, the free-boundary $r=\gamma(\theta,t)$ and the
$\e$-level set $ r=\gamma_\e(\theta,t)$ of $f$ for each $\e > 0$ move with finite and nondegenerate speed on $0 \leq t \leq T$.
 \end{lem}

\subsection{Gradient Estimates}\label{sec_3.2}
Throughout this section, we will assume that $g=(\beta f)^{\frac{1}{\beta}}$ is solution of $\eqref{eqn-g}$ and smooth up to the interface on $0 \leq t \leq T$ and also is satisfied with
\begin{equation} \label{eqn-3.1}
\max_{x \in \varOmega(t)}  g(x,t) \geq 2, \qquad \text{for}
\,\, 0\leq t \leq T,
\end{equation}
which comes from \eqref{maxf}.
We now will show that the gradient $|Dg|$ has the bound from above and below.

\begin{lem}[Optimal Gradient estimates]\label{lem-31}
With the same assumptions of Theorem \ref{thm-1.3}  and \eqref{eqn-3.1},
there is a positive constant $C_0 $  such that
$$|Dg| \leq C_0\: , \, \qquad \text{on}\,\,\,   0 \leq g(\cdot,t) \leq 1, \,\,
0\leq t
\leq T.$$
Moreover if \eqref{co1} is satisfied and 
if $g$ is smooth up to
the interface on $0 \leq t \leq T$, then there is a positive constant $c_0 $
such that
$$|Dg| \geq c_0 \: , \qquad \text{on \, } g(\cdot,t) >0, \,
0 \leq t \leq T.$$\end{lem}
\begin{proof}
(i) First, we are going to show the upper bound of $\D g$. Suppose that $f$ is approximated by $f_\e$ of \eqref{eqn-f} which is a decreasing sequence of solutions satisfying the positivity, strictly convexity and smoothness on $\{\, x \in
\Real^2:   |Df_\e(x)| < \infty \, \}$ for $0 \leq t \leq T$. Set
$g_\e = (\beta f_\e)^{\frac{1}{\beta}}$. We can choose the $f_\e's$ such
that $|Dg_\e| \leq C_0$ at $t=0$, on the set $\{ \, x: \, 0 \leq
g_\e \leq 1 \, \}$ and $|D g_\e|  \leq C_0$ at $g_\e=1$, $0 \leq t
\leq T$, for some uniform constant $C_0$. The last estimate holds
because of \eqref{maxf} and \eqref{eqn-3.1}.

Let's denote $g_\e$ by $g$ for convenience of notation, where $g=(\beta f)^{\frac{1}{\beta}}$ is a strictly positive  and a smooth solution of \eqref{eqn-g} with convex $f$. 
Let us apply the maximum principle to $X=  \frac{|Dg|^2 }{2}= \frac{g_x^2+g_y^2}2$ and assume $X$ has an interior maximum at the point $P_0=(x_0,y_0,t_0)$. By rotating the coordinates, we can assume $g_x > 0$ and $g_y=0$ at $P_0$.
Then we have $ X_t \leq 0$ by using the facts that $X_x =X_y =0$,
$X_{xx}  \leq 0$ and $X_{yy}  \leq 0$ are satisfied at $P_0$.
On the other hand  $|\D g|$ is bounded at $t=0$ from the condition on the initial data and on $\{g=1\}$, $|\D g|=\frac{|\D f|}{g^{\beta-1}}=|\D f|$ is bounded since $f$ is convex. Hence $X \leq \tilde C$, on  $0 \leq g
\leq 1, \,\, 0\leq t \leq T$, provided that $X \leq \tilde C$ at $t=0$ and $g=1$, $0 \leq t \leq T$ so that 
$$|Dg| \leq C_0, \, \qquad \text{on}\,\,\,   0 \leq g(\cdot,t) \leq 1,
\,\, 0\leq t \leq T.$$
(ii.) Now we are going to show the lower bound of the gradient.
Consider $$X=x\, g_x + y\, g_y.$$ Using the
maximum principle as (i),  we have that
\begin{equation}\label{eqn-X1}
X_t \geq -C \, X
\end{equation}
where $C$ is a constant depending on $\rho_0$ and
\begin{equation}
\frac{d}{dt} \, X(\gamma(t),t)  \geq - C \, X
\end{equation}
at a interior or boundary
minimum  point $P_0$ of $X$. Then 
$$\min_{ \{ g(\cdot,t) >0 \}} X(t) \geq  \min_{ \{ g(\cdot,0) >0 \}} X(0) \, e^{-Ct}$$
for all $0 \leq t \leq T$ by Gronwall's inequality, and it implies the desired estimate.
\end{proof}

\begin{thm}\label{thm-33}
Under the same assumptions of Lemma \ref{lem-31},  there exist positive
constants $C_1$, $C_2$ and $\e_0$,
 depending only on $\rho_0$
and the initial data, for which
\begin{equation}\label{eqn-i13}
-C_2 \leq (\gamma_\e)_t (\theta,t)  \leq -C_1 < 0, \qquad  \text{for} \,\, 0\leq
t \leq T \,\, \text{and }\,\,0<\e<\e_0.
\end{equation}
\end{thm}
\medskip

\section{Second derivative estimate}
\subsection{ Decay Rate of $\alpha$-Gauss Curvature} Under the same conditions with section (\ref{sec_3.1}) and (\ref{sec_3.2}), we will show a priori bounds of the Gauss Curvature $K= \det (D^2f) /  (1+|Df|^2)$ and the second derivatives of $f$ and $g$. 

\begin{lem}\label{lem-41} For the same hypothesis of Theorem \ref{thm-1.3} and \eqref{co1}, there
exists a positive constant  $c $ such that 
\begin{equation}\label{eqn-j2}
 c \leq \frac {K^{\al}}{g^{\frac{1}{2\alpha-1}}} \leq c^{-1}, \qquad \text{on} \,\,\,  0 \leq t \leq T
\end{equation}
for $K= \det D^2f /  (1+|Df|^2)$.
\end{lem}

\begin{proof}
We will only consider the bound of \eqref{eqn-j2} around the interface. It suffices to show the bound of $g_t$ from $g_t=
K^{\alpha} /\left((1+|Df|^2)^{\frac{2\al -1}{2}} g^{\frac{\alpha}{2\alpha-1}} \right)$ because $|Df|$ is bounded around $\{ g=0\}$.
For $r=\gamma_\e(\theta,t)$ which is the $\e$-level set of $g$ in polar coordinates,
$$ g_t = -  g_r \cdot  \dot \gamma_\e (\theta,t)$$ since $g(\gamma_\e(\theta,t), \theta, t) =\e$ and  the level sets of $g$ is convex.
Then we know that $ c  < g_r < c^{-1} $ and $-C_2 \leq \dot \gamma_\e  (\theta,t)  \leq -C_1 < 0$ for $ 0\leq t \leq T$ from Lemma \ref{lem-31} and Theorem  \ref{thm-33} implying that $ C_1 \, c <  g_t 
<  C_2  \, c^{-1}$, so proof is completed.
\end{proof}

\begin{corollary}\label{cor-42}
Under the assumptions of Lemma $\ref{lem-41}$, 
the solution $g$
of \eqref{eqn-g} satisfies the bound
\begin{equation}\label{eqn-j3}
 c \leq g_t \leq c^{-1}.
\end{equation}
\end{corollary}

\subsection{Upper bound of the Curvature of Level Sets}

\begin{lem}\label{lem-43}  With the assumptions of Theorem \ref{thm-1.3} and condition
\eqref{co1},  there exists a constant  $C>0 $ such that
$$ 0 < g_{\tau\tau} \leq C$$
with $\tau$  denoting the  tangential direction  to the level sets
of $g$.
\end{lem}

\begin{proof}
Strictly convexity of the level sets of $g$ directly implies $g_{\tau\tau} >0 $. We will obtain the bound from above by using the maximum principle on
\begin{equation}\label{eqn-j300}
X = g_y^2 g_{xx} - 2g_xg_yg_{xy} + g_x^2g_{yy} + (g(g_{xx}+g_{yy})+\theta |\D g|^2).
\end{equation}
Let $\nu$ and $\tau$ denote the outward normal and tangential
direction to the level sets of $g$ respectively. Then we can write $X$ as\begin{equation}\label{eqn-j30}
X  = (g+ g_{\nu}^2) \, g_{\tau\tau} + (g\, g_{\nu\nu} + \theta g_\nu^2)
\end{equation}
since $g_\tau =0$.
We have also known that
$$ 0 < c \leq  g_{\nu} \leq c^{-1} \qquad \text{on} \,\, g> 0, \, \, 0 \leq t \leq
T$$ for some $c >0$, depending on $\rho_0$ and the initial data.
Also $ g(g_{xx}+g_{yy})+\theta |\D g|^2 $ is  bounded since $f \in{C^{1,1}}$.
Therefore, an upper bound on $X$ will imply the desired upper bound on
$g_{\tau\tau}$.  We will apply the maximum principle on the evolution of $X$. The term $(g(g_{xx}+g_{yy})+\theta |\D g|^2)$ on $X$ will control the sign of error terms.
Corollary \ref{cor-42} implies
$$X \leq C \qquad \text{at} \quad g=0,$$
since we know that $X=\frac{1}{\theta}g_t^{\frac{1}{\alpha}} +\theta |\D g|^2$ at the free-boundary $g=0$. 
Then we can assume that $X$ has its space-time maximum at
an interior point $P_0=(x_0,y_0,t_0)$. Let's assume that
\begin{equation}\label{eqn-j31}
g_\tau= g_y=0 \,  \qquad \mbox{and} \qquad g_\nu =
g_x >0\quad
\mbox{at \, $P_0$}
\end{equation}
without loss of generality, since $X$ is
rotationally invariant. Also let's consider the following transformation
$$ \tilde{g} (x , y ) = g(\mu , \eta  ) $$
where $\mu = x$ and $\eta=y-ax $ with $ a = \displaystyle \frac{g_{\mu \eta}(x_0 , y_0 , t_0)}{g_{\mu \mu } (x_0 , y_0 , t_0)}$.
Then we can obtain $\tilde{g}_x = g_{\mu} - \frac{g_{\eta}g_{\mu \eta}}{g_{\eta \eta}}=g_{\mu} >0 $, \, $\tilde{g}_y = g_{\eta} =0 $, and 
\begin{equation*}
\begin{split}
(\tilde{g}_{ij})= \left[
\begin{array}{clc}
g_{\mu \mu} -\displaystyle \frac{g_{\mu \eta} ^2}{g_{\eta \eta}} &\quad\quad  0 \\\\
0&\quad \quad g_{\eta \eta} \\
\end{array}
\right]
\end{split}
\end{equation*}
at $ P_0 $. 
Here $\tilde{g}_{yy} =g_{\eta \eta} >0 $ and $\tilde{g}_{xx}<0$ at $ P_0 $. Hence equation is maintained with this change of coordinate. Also we can drop off the third derivative term of $\tilde{g}$ because it is changed under the perfect square of the third derivative of $g$. Hence we can assume 
\begin{equation}\label{eqn-j31-1}
\tilde{g}_{xy }=0 
\end{equation}
at $P_0$ without loss of generality. Then we will proceed with the function $g$ instead of $\tilde{g}$ for convenient of notation.
From \eqref{eqn-j300}, we get
$$X=(g+g_x^2) g_{yy} +  (g g_{xx} + \theta g_x^2), \quad
\mbox{at \, $P_0$}.$$ 
At the maximum point $P_0$,  we also have
$X_x = 0$ and $X_y=0$
implying that
\begin{equation}\label{eqn-j4}
g_{xyy} = - \frac {g  g_{xxx} +  2 g_x \det D^2 g
+ (2\theta +1) g_xg_{xx} + g_x g_{yy}}{g+ g_x^2}
\quad \text{and }\quad g_{yyy} = - \frac{ g g_{xxy} }{g + g_x^2}.
\end{equation}

We next compute the evolution equation of $X$ from the evolution equation
of $g$
to find a contradiction  saying that  $$0\leq X_t <0 \qquad \mbox{at \,\, $P_0$},$$  when $X>C>0$ for some
constant
$C$.  This implies that $X \leq C$, on $0 \leq t \leq T$.

First we will consider the following simpler case that $f$ satisfies the evolution
$$f_t =( \det D^2 f)^{\alpha}$$
for the convenience of the reader. Then $g = (\beta f)^{\frac{1}{\beta}}$ satisfies the equation
\begin{equation}\label{eqn-j6}
g_t  = (g \, \det D^2 g + \theta(g_y^2g_{xx} - 2g_xg_yg_{xy}+g_x^2g_{yy}))^{\alpha}.
\end{equation}
To compute the evolution of $X$ we differentiate twice the equation
\eqref{eqn-j6}. 
Set
\begin{align*}
K_{g} &=g \, \det D^2 g + \theta(g_y^2g_{xx} - 2g_xg_yg_{xy}+g_x^2g_{yy}),\\
I &=1+g^{2\beta-2}|\D g|^{2} ,  \,\,\text{and }\,\, J =g+|\D g|^{2}. 
\end{align*}
Let $L$ denote the operator
$$LX := X_t - \alpha K_{g}^{\alpha-1}\{  \, (gg_{yy}+ \theta g_y^2)  \, X_{xx} - 2 (gg_{xy} + \theta g_xg_y) \, X_{xy} +(gg_{xx}+ \theta g_x^2) \, X_{yy} \, \}.
$$
Then after many tedious calculations, we have that at the maximum point
$P_0$, 
\begin{equation}\label{eqn-j7}
  \begin{split}
LX =A+\frac{1}{(1+2\gamma)^2 (g+g_x ^2)^2 K_g^2}B
    \end{split}
      \end{equation}
      
 where $\gamma = \theta -1$
 and 
  \begin{equation}\label{eqn-j8}
  \begin{split}
A&=- 4g^2 g_{xxy} ^2 -\frac {4 g^3}{g^2_x+g}\, \Big(g_{xxx}+\frac{6g_x g_{xx}+3g_{x}g_{xx}g_{yy}}{g}\Big)^2\\
&+\frac{g^2_{yy}}{g^2_x+g}\big{\{}-(2g^2_x-gg_{xx})^2 +3g_{xx}(g^2_x+g)(g^2_x-gg_{xx})\big{\}}.
    \end{split}
      \end{equation}
In addition, $B=0$ if $\gamma=0$, otherwise   
  \begin{equation}\label{eqn-j9}
  \begin{split}
B & =-B_1 g^2 (g^2_x+g) g^2_{xxy}-g^3B_1(g_{xxx}+B_{11})^2 +\big((1+2\gamma)^2 (g+g_x ^2)^2 K_g^2\big) \frac{E_1 }{E_2 }.
    \end{split}
      \end{equation}
Here 
 \begin{equation*}
  \begin{split}
B_1&=4(1+2\gamma)^2K_g^{\frac{2+3\gamma}{1+2\gamma}}\big(\frac{1+\gamma}{1+2\gamma}-K_g^{\frac{\gamma}{1+2\gamma}}\big)+\gamma(1+\gamma)K_g^{\frac{1+\gamma}{1+2\gamma}}\big( (g_x ^2 +g)g_{yy}-\big{\{}gg_{xx}+(1+\gamma)g_x ^2 \big{\}}\big)^2
    \end{split}
      \end{equation*}
      and set  $Z=g_x ^2 g_{yy}$ so that
 \begin{equation*}
  \begin{split}
   E_2 &=4(1+2\gamma )^3g g^6_x (g^2_x+g)^2 K_g^{\frac{6+13\gamma}{1+2\gamma}}Z^4 \Big( \frac{(1+\gamma)(1+\frac{3}{2}\gamma)}{(1+2\gamma)^2}- K_g ^{\frac{\gamma}{1+2\gamma}} \Big)\\
 &+g\gamma(1+\gamma)(1+2\gamma)g^2_x(g^2_x+g)^3 K_g^{\frac{5+11\gamma}{1+2\gamma}}Z^6+l.o.t.\\
&\geq g(E_{11}Z^4 + \gamma E_{12}Z^6 + l.o.t.).
    \end{split}
      \end{equation*}
where $l.o.t.$ means lower order term. We may assume that $P_0 $ lies close to the free-boundary and that  $K_g ^{\frac{\gamma}{1+2\gamma}}<\frac{1}{2}$ by considering a scaled solution $g_{\lambda}(x, t)={\lambda}^{-\frac{1}{\gamma+2}}g(\lambda x, \lambda^{\frac{4\gamma+3}{2\gamma +1}}t)$ as  $g$ at the beginning of proof with $\lambda^{\frac{4\gamma +5}{\gamma +2}} \leq \frac{1}{||K_g||_{L^{\infty}}}\big(\displaystyle \frac{1}{2}\big)^{\frac{1+\gamma}{\gamma}}$. Then on $g \leq 1$, we have $A$ is negative in \eqref{eqn-j8} since $g_{xx}$ is negative and $E_{11}, E_{12} \geq \delta _0 (g_x ,  K_g) > 0$ uniformly, which implies $E_2$ is positive. And we also have,  in \eqref{eqn-j9},
 \begin{equation*}    
 \begin{split}
   E_1 &=-\gamma (1+\gamma ) (g^2_x+g)^2 K_g^5Z^8 C_8 + \gamma(Z^7 C_7 + l.o.t.)
   \end{split}
       \end{equation*}
 \begin{equation*}   
 \begin{split}
\text{with} \,\, C_8 &=\Big( (1+\gamma )^2 (g^2_x+g)(2\gamma g+3(5+4\gamma)g^2_x)+ \gamma (1+2\gamma )  K_g ^{\frac{\gamma}{1+2\gamma}} g^2 \\
 &-(1+2\gamma)(15+2\gamma(7+2\gamma)) K_g^{\frac{\gamma}{1+2\gamma}}g g^2_x -3(1+\gamma)(1+2\gamma)(5+4\gamma)K_g^{\frac{\gamma}{1+2\gamma} }g^4_x \Big).
    \end{split}
      \end{equation*}
Now  we can show  $C_8 \geq \delta _1 (g_x ,  K_g)  > 0$ uniformly and then  $E_1< 0$ for sufficiently large $Z$.  Therefore $B$ is negative. Hence we can obtain desired result.

We now return to the case of the $\alpha$-Gauss Curvature Flow. 
Let us set $$I = 1 + g^{2\beta-2} \, |Dg|^2, \, J= g+ |Dg|^2 \, \text{and} \, Q=(g \det D^2 g + \theta(g_y^2g_{xx} - 2g_xg_yg_{xy}+g_x^2g_{yy}))^{\alpha}.$$ Also let $C=C(\|g\|_{C^1},\, \|f \|_{C^{1,1}})$ denote various constants and $\tilde{L}X $ denote the operator
$$\tilde{L}X := X_t - \alpha K_{g}^{\alpha-1} I^{-\frac{4\alpha-1}{2} }\{  \, (gg_{yy}+\theta g_y^2)  \, X_{xx} - 2 (gg_{xy} + \theta g_xg_y) \, X_{xy} +
(gg_{xx}+ \theta g_x^2) \, X_{yy} \, \}.$$
We find, after several calculations, that at the maximum point $P_0$,
where \eqref{eqn-j31} and \eqref{eqn-j31-1} hold, $X$ satisfies
the inequality
 \begin{equation*}
  \begin{split}
\tilde{L}X &=I^{-\frac{4\alpha-1}{2} }LX -\frac{4\alpha -1}{2} (g+ g_x ^2 )I^{-\frac{4\alpha+1}{2}} Q I_{yy} -(4\alpha -1 )g_x (\theta + g_{yy}) I^{-\frac{4\alpha+1}{2}} Q I_x \\
&+g\Big{\{} -(4\alpha -1) I^{-\frac{4\alpha+1}{2}} Q_x I_{x} - \frac{4\alpha -1}{2} I^{-\frac{4\alpha+1}{2}}  Q I_{xx} + \frac{16\alpha ^2 -1}{4} I^{-\frac{4\alpha+3}{2}} Q I_{x} ^2 \Big{\}} + I^{-\frac{4\alpha+1}{2}} Q g_{xx} 
 \end{split}
      \end{equation*}
and from \eqref{eqn-j7} and tedious computation we obtain that
 \begin{equation*}
  \begin{split}
\tilde{L}X &\leq I^{-\frac{4\alpha-1}{2} }A +\frac{I^{-\frac{4\alpha -1}{2}}}{(1+2\gamma )^2(g+ g_x ^2 )^2K_g ^2}\Big{ \{ }-B_1 g^2 (g^2_x+g) g^2_{xxy}-g^3B_1(g_{xxx}+B_{11})^2 \\
&\quad+\big((1+2\gamma)^2 (g+g_x ^2)^2 K_g^2\big) \frac{E_1 }{E_2 }\Big{  \} } \\
&\quad -\frac{8\gamma (\gamma +1)}{(1+2\gamma )^2} I^{-\frac{4\alpha+1}{2}} K_g ^{\alpha -1}g^{2\gamma +3}\frac{g_x}{g+g_x ^2 }\Big{\{} ((\gamma +1)g_x ^2 + gg_{xx})^2-(g+g_x ^2)K_g \Big{\}}g_{xxx}\\
&\quad  + g^{2\gamma +1}\cdot l.o.t. +  I^{-\frac{4\alpha+1}{2}} K_g ^{\alpha} g_{xx} \\
&\leq I^{-\frac{4\alpha-1}{2} }A +\frac{I^{-\frac{4\alpha -1}{2}}}{(1+2\gamma )^2(g+ g_x ^2 )^2 K_g ^2}\Big{ \{ }-B_1g^2 (g^2_x+g) g_{xxy} ^2 -g^3 B_1 (g_{xxx}+B_{11} + O(g))^2\\
&\quad +\big((1+2\gamma)^2 (g+g_x ^2)^2 K_g^2\big) \frac{E_1}{E_2} + O(g)\Big{  \} } + O(g) +  I^{-\frac{4\alpha+1}{2}} K_g ^{\alpha} g_{xx} .
 \end{split}
      \end{equation*}
Here $O(g)$ denotes various terms satisfying $|O(g)| \leq Cg$ with constant $C$. We can know the first term and the second term are negative as in the case of $LX$ and provided that $X \geq C$ is sufficiently large. And then $\tilde{L}X\leq C$ with $C$ depending on $||f||_{C^{1,1}}$ and $||g||_{C^1}$ on $g \leq 1$, which implies that $(X-Ct)_t \leq 0$. Applying the evolution of $\tilde{X}=X-Ct$  with a simple trick implies $
\tilde{X} \leq C$ where $C$ is positive constant. This immediately gives the desired contradiction.
\end{proof}

\subsection{Aronson-B\'{e}nilan type Estimate}
\begin{lem}\label{lem-AB}
Under the assumptions of Theorem \ref{thm-1.4} and condition
\eqref{co1},  there exists a constant  $C >0$
for which
$$ \det (D^2 g) \geq -C$$
for a uniform constant $C>0$.
\end{lem}

\begin{proof}
To establish the bound of  $ \det (D^2 g) $ from below, we will use the maximum principle on
the quantity
$$Z=\frac{\det D^2 g}{g_x ^2 g_{yy} + g_y ^2 g_{xx}-2g_x g_y g_{xy}}+ b\mid \nabla g \mid ^2$$
with some positive constant  $b$ on $\{ \, g(\cdot,t) >0, \, 0
\leq t \leq T \, \}.$
Let us assume that $Z$ becomes minimum at the interior point $P_0$.
We can assume $g_y = 0 $, $g_x >0$ and $g_{xy}=0$ at $P_0 $ by using similar transformation and the change of coordinate in Lemma \ref{lem-43} at $P_0$.
Then we have
\begin{equation*}
\begin{split}
a_{ij}Z_{ij} \leq 0
\end{split}
\end{equation*}
for
\begin{equation}\label{diffusion-coeff}
\begin{split}
(a_{ij})= \left[
\begin{array}{clc}
\alpha K_g^{\alpha-1}gg_{yy}(1+g^{\frac{2\alpha}{2\alpha -1}}g^2_x)^{\frac{1-4\alpha}{2}} & \quad\quad\quad \quad\quad\quad\quad  0 \\
0&\alpha K_g^{\alpha-1}(gg_{xx}+\theta g^2_x)(1+g^{\frac{2\alpha}{2\alpha -1}}g^2_x)^{\frac{1-4\alpha}{2}}\\
\end{array}
\right]
\end{split}
\end{equation}
and $Z_x = Z_y = 0$ at the minimum point $P_0$ implying that \\
\begin{equation}\label{eqn-Z}
\begin{split}
Z_t &\geq B_1 (g_{xyy}+B_2)^2 + A_0 + \frac{1}{(\al +1)(2\al -1)^3}O(g)Z + O(Z^2) \\
\end{split}
\end{equation}
where 
\begin{equation*}
\begin{split}
 B_1=\frac{\al}{(1-2\al)^2 g_x ^2 g_{yy}}&\big( 1+g^{\frac{2\al}{2\alpha -1}} g_x ^2 \big)^{\frac{1}{2}-2\alpha} K_g ^{\alpha -2}\big[ \al g_x ^2 + (2\al -1)g g_{xx}\big]\\
\cdot &\Big[ (4\alpha -2)K_g + (\alpha -1)\big{\{}\alpha g_x ^2 +(2\al -1)gg_{xx}\big{\}}g_{yy} \Big]
 \end{split}
\end{equation*}
 and 
 \begin{equation*}
\begin{split}
A_0 = A_{0,0}g^2 b^4+E_0+A_{0,5}(\al-1)g^3 b^5 + (\al-1)E_1
 \end{split}
\end{equation*}
with
\begin{equation}\label{eqn-A00}
\begin{split}
A_{0,0}=\frac{g_x ^{10} g _{yy}(24+12 g^2  g_x ^2 -21 g^4  g_x ^4)}{2(1+g^2  g_x ^2 )^{7/2}}
\end{split}
\end{equation}
and
\begin{equation*}
\begin{split}
A_{0,5}=-\frac{30 g^{1+\frac{1}{2\alpha -1}}(1-2\alpha)^2 g_x ^{14}(g^{\frac{2\alpha}{1-2\alpha}} + g_x ^2)(1+g^{1+\frac{1}{2\alpha -1}} g_x ^2)^{-\frac{1}{2}-2\alpha} K_g ^{\alpha -1}g^2 _{yy}}{(4\alpha -2)K_g + (\alpha -1)\alpha g_x ^2 g_{yy}}
\end{split}
\end{equation*}
where $E_0=O(b^3, g^2)$ and $E_1=O(b^4, g^3)$.
Here we can also show  $A_{0,0} \geq \delta _1 (g_x ,  g_{yy}) > 0$  uniformly and we have
\begin{equation}\label{denom}
\begin{split}
(4\alpha -2)K_g + (\alpha -1)\alpha g_x ^2 g_{yy} &=(4\alpha -2)(gg_{xx}g_{yy}+\theta g_x ^2 g_{yy})+(\alpha -1)\alpha g_x ^2 g_{yy} \\
&=(4\alpha -2)gg_{xx}g_{yy} + \{\theta (4\alpha -2)+\alpha ^2 -\alpha\}g_x ^2 g_{yy}\\
&>0  
\end{split}
\end{equation}
on $g \leq 1$ since $\displaystyle \frac{1}{2} < \alpha \leq 1$. Then $(\al -1)A_{0,5} $ is nonnegative so that $A_0$ is positive for sufficiently large $b\gg 1$. Also we can know $B_1$ is positive on $g \leq 1$ from \eqref{denom}.
This implies $$Z_t > 0 > -aZ$$ with a positive constant $a $.
By Gronwall's inequality, we have  $$Z\geq Z_0 e ^{-\tilde{a} t} $$
where $Z_0 $ is initial data of $Z $ at $ t=0$ and $\tilde{a}$ is constant, which concludes the proof.
\end{proof}

\subsection{Global Optimal Regularity}
Let's consider the quantity
\begin{equation}\label{eqn-43-1}
\cZ=\max_\gamma(gD_{\gamma\gamma}g + \theta|D_\gamma g|^2).
\end{equation}
Now we will show that $\cZ$ is bounded from above through next lemma.
\begin{lem}\label{lem-43-1}
With the same assumptions of Theorem \ref{thm-1.4} and condition
\eqref{co1},  there exists a positive constant  $C=C(\theta,\rho,\lambda,\| g\|_{C^2 ({\partial \Omega})}) $
with
$$\max_{\varOmega(g)}  \cZ \leq C$$
where $\varOmega(g) = \{x|g(x) > 0\}$.
\end{lem}

\begin{proof}
First, we know that $\cZ$ is nonnegative from $\cZ=\beta^{\frac{2-\beta}{\beta}}f^{\frac{2-\beta}{\beta}} f_{\gamma \gamma}$ and a convexity of $f$. Also Lemma \ref{lem-31} implies 
$$ \cZ \leq C(\theta,\rho,\lambda,\| g\|_{C^2 ({\partial \Omega})}) \quad \text{at} \,\, g=0,$$
since $\cZ=\theta|D_\gamma g|^2$ at the free-boundary $g=0$.
Then we can assume that $\cZ$ has the its maximum at an interior point $P_0 \in \Omega (g)$ and at a direction $\gamma$. To show the bound of $\cZ$, we consider $\gamma$ as $\gamma=\lambda _1 \nu + \lambda _2 \tau$ with $\lambda ^2_1 + \lambda ^2_2 =1$, where $\nu , \, \tau$ denote the outward normal and tangential directions to the level sets of $g$ respectively. Then $\cZ(P_0)=gD_{\gamma\gamma}g + \theta|D_\gamma g|^2 $ and \eqref{eqn-g} can be rewritten as 
\begin{equation}\label{eqn-43-max}
\begin{split}
 &\cZ(P_0)=g[\lambda ^2_1 g_{\nu \nu}+2\lambda _1 \lambda _2 g_{\nu \tau}+\lambda ^2_2 g_{\tau \tau}]+\lambda ^2_1 g^2_{\nu}\\
 \text{and }\,\, &(gg_{\nu \nu}+\theta g^2_{\nu})g_{\tau \tau}=gg^2_{\nu \tau}+\{g_t (1+g^{2\beta -2}g^2_{\nu})^{\frac{4\al -1}{2}}\}^{\frac{1}{\alpha}}.
\end{split}
\end{equation}
Here if $gg_{\nu \nu}$ is not sufficiently large at $P_0$, we have 
$$\cZ(P_0)\leq C(\theta,\rho,\lambda,\| g\|_{C^2 ({\partial \Omega})})$$
from Lemmas \ref{lem-31}, \ref{cor-42} and \ref{lem-43} implying the desired result immediately.
On the other hand, if $\theta g^2_{\nu} \leq gg_{\nu \nu}$ at $P_0$, then we get $$gg^2_{\nu \tau} \leq 2gg_{\nu \nu}g_{\tau \tau} \leq C(\theta,\rho,\lambda,\| g\|_{C^2
({\partial \Omega})})gg_{\nu \nu}$$
implying $g_{\nu \tau} \leq C(\theta,\rho,\lambda,\| g\|_{C^2({\partial \Omega})}) \sqrt{g_{\nu \nu}}$. Then we can know that $\cZ(P_0)$ is maximum when $\lambda _2 =0$ from Lemma \ref{lem-43} and \eqref{eqn-43-max}
so that $\cZ(P_0)=gg_{\nu \nu} + \theta g^2_{\nu}$. Also we get $gg_{\nu \tau}+\theta g_{\nu} g_{\tau}=0$ at $P_0$ %since $\nu$ is the maximum direction, that , 
implying that $g_{\nu \tau}=0$ at $P_0$.
 Here by the similar transformation in Lemma \ref{lem-43}, we can assume 
\begin{equation*}\label{eqn-j311}
g_\tau= g_y=0 , \quad g_\nu =
g_x >0 \quad \mbox{and} \quad g_{xy}=0 \quad \mbox{at \, $P_0$} .
\end{equation*}
Then we have
$$a_{ij}\cZ_{ij} \leq 0$$
with \eqref{diffusion-coeff}
at the maximum point $P_0$. And since $g_{xxx}=\frac{-g_x g_{xx}-2\theta g_{x}g_{xx}}{g}$ and $g_{xxy}=0$ at $P_0$, we obtain
\begin{equation*}
\begin{split}
\cZ_t &= g_t g_{xx}+ gg_{xxt} + 2\theta g_x g_{xt}\\
&\leq -(1+g^{\frac{2\alpha}{2\alpha-1}}g^2_x)^{-\frac{1+4\alpha}{2}}[(1-2\alpha)^2(\alpha-1)(g^{\frac{1}{1-2\alpha}}+gg^2_x)]^{-1}[g_{yy}(\theta g^2_x +gg_{xx})]^{\alpha}\\
&\quad \cdot [g^{\frac{2\alpha}{2\alpha-1}}\alpha(4\alpha-1)(2\alpha (2\alpha \theta -3\theta +\alpha +1)+2\theta -1)g^6_x \\
&\quad+ (1-2\alpha)^2(\alpha-1)g_{xx}\{g^{\frac{1}{1-2\alpha}}(\alpha-1)+(4\alpha-1)g^2 g_{xx}\} \\
&\quad+ (2\alpha-1)gg^2_{x}g_{xx}\{(\alpha-1)(16\alpha ^2-5\alpha +1)+ 3\alpha (8\alpha^2 -6\alpha +1)g^{\frac{4\alpha-1}{2\alpha-1}}g_{xx}\}\\
&\quad+\alpha g^4_x\{(4\alpha^2-5\alpha+1)(\theta (4\alpha -2)+1)+6\alpha(10\alpha^2-9\alpha +2)g^{\frac{4\alpha-1}{2\alpha-1}}g_{xx} \}]
\end{split}
\end{equation*}
at the point $P_0$. Also from 
$g_{xx}=\frac{\cZ-\theta g^2_x}{g}$, we have
\begin{equation*}
\begin{split}
\cZ_t &\leq \frac{(1+g^{\frac{2\alpha}{2\alpha-1}}g^2_x)^{-\frac{3+4\alpha}{2}}(\cZ g_{yy})^{\alpha} }{g(\alpha-1)(2\alpha-1)}\Bigg[-\cZ(2\alpha^2-3\alpha +1)\big{\{}(4\alpha-1)\cZ g^{\frac{2\alpha}{2\alpha-1}}+\alpha-1\big{\}}\\
&\quad+g^{\frac{4\alpha}{2\alpha-1}}g^2_x\big{\{}g^{\frac{4\alpha}{1-2\alpha}}(\alpha-1)^2 \alpha -3\alpha (8\alpha^2 -6\alpha +1)\cZ^2 - (\alpha-1)(8\alpha^2-3\alpha +1)\cZ g^{\frac{2\alpha}{1-2\alpha}}\\
&\quad+\alpha (\alpha -1)g^2_x(2(\alpha-1)g^{\frac{2\alpha}{1-2\alpha}}-6\alpha \cZ + (\alpha-1)g^2_x)\big{\}}\Bigg]\\
&\leq (1-4\alpha)(1+g^{\frac{2\alpha}{2\alpha-1}}g^2_x)^{-\frac{3+4\alpha}{2}} g_{yy} ^{\alpha}  \cZ^{2+\al }g^{\frac{1}{2\alpha-1}}+O(\cZ ^{2+\al})g^{\frac{2\alpha+1}{2\alpha-1}}+O(\cZ^{1+\al },g^{\frac{1}{2\alpha-1}}).
\end{split}
\end{equation*}
Then on $g\leq 1$, $\cZ_t \leq 0 $ at $P_0$ since $1-4\alpha <0$. Hence we can obtain the desired result.
\end{proof}

\subsection{ Decay rates of Second Derivatives}

\begin{corollary}\label{cor-44}
With the same assumptions of Lemma \ref{lem-43}, there exist a positive constant $c =c(\rho_0,f_0)$  such that
\begin{enumerate}
\item $ c \leq  g_{\tau\tau} \leq c^{-1} $
\item $c\, \leq f_{\nu\nu},\, \frac{f_{\tau\tau}}{g^{\beta-1}}\, \leq   c^{-1}
\quad \text{and}\,\frac{|f_{\nu\tau}|}{ g^{(\beta-1)/2}}\leq   c^{-1}$
\end{enumerate}

with $\tau$ denoting the tangential direction to the level sets of $g$.
\end{corollary}

\begin{proof}
(i.) The upper bound of $g_{\tau\tau}$ comes from Lemma \ref{lem-43}.
Now we are going to show the lower bound.
From Lemma \ref{lem-41}, we have
$$\det D^2 f \geq c \, g^\frac{1}{2\al-1}= c\,g^{2\beta-3}.$$
which implies
\begin{equation}
f_{\nu\nu}f_{\tau\tau} \geq c\, g^{2\beta-3}  + f_{\nu\tau}^2 \geq
c \, g^{2\beta-3}
\end{equation}\label{eqn-4.9}
and then
$$f_{\tau\tau} \geq \frac {c\, g^{2\beta-3}}{f_{\nu\nu}} \geq \tilde c \, g^{\beta-1}$$
since $f_{\nu\nu} \leq Cg^{\beta-2}$ from Lemma \ref{lem-43-1}.
Since $f_{\tau\tau} = g^{\beta-1}\,
g_{\tau\tau} + (\beta-1)g^{\beta-2}\, g_\tau^2 $, we conclude that
$$g_{\tau\tau} = \frac{f_{\tau\tau}}{g^{\beta-1}} \geq \tilde c,$$ for some
positive constant $\tilde c$ depending only on the initial data
and $\rho_0$.\\
(ii.) $f_{\tau\tau}= g^{\beta-1}\, g_{\tau\tau}$ and the bound on $g_{\tau\tau}$ tell us
$$ c \leq \frac{f_{\tau\tau}}{g^{\beta-1}} \leq c^{-1}.$$
(iii.) Third, we are going to show $$f_{\nu\nu} + f_{\tau\tau} \geq c. $$
 Let us denote by
$\lambda_1, \lambda_2$ the two eigenvalues of the matrix $\det D^2
f$ such that $\lambda_1 \geq \lambda_2$. Then, from Lemma \ref{lem-41}, we have
$$ c \leq \frac{\lambda_1 \lambda_2} {g^{2\beta-3}} \leq c^{-1}$$
and  $\lambda_2 \leq f_{\tau\tau} \leq c^{-1} \, g^{\beta-1}$,
implying that
$ c^{-1} g^{\beta-2} \geq \lambda_1 \geq c g^{\beta-2}$
for some positive constant $c$ since $ 2< \beta$. Hence,
$f_{\nu\nu} + f_{\tau\tau} \geq \lambda_1+\lambda_2 \geq c >0$
as desired.\\
(iv.) From (i) and Lemma \ref{lem-43-1}, we have
$$c\leq f_{\nu\nu}+ g^{\beta-1}\, g_{\tau\tau}\leq 2 f_{\nu\nu}\leq \sup |D^2f|<C$$
for a uniformly small $g$, which is true in a uniformly small neighborhood of free boundary. The convexity of $f$ says $0\leq f_{\nu\tau}^2 - f_{\nu\nu}\, f_{\tau\tau}$.  By using the bound
$f_{\nu\nu} \leq c^{-1}$,
 we obtain
$$f_{\nu\tau}^2 \leq f_{\nu\nu} \, f_{\tau\tau} \leq c^{-1} \, g^{\beta-1}$$
\end{proof}

\section{Higher Regularity}
\subsection{Local Change of Coordinates}
For any point $P_0= P_0(x_0,y_0,t_0)$ at the interface $\Gamma$ with $0 < t_0 \leq T$, let's assume that $n_0$ is the unit vector in the direction of the vector ${\bf P_0} =
\overline{OP_0}$ and $n_0$ satisfies
\begin{equation}\label{nvec}
n_0:=   \frac{\bf P_0}{|{\bf P_0}|}  = {\bf e_1}
\end{equation} 
by rotating the coordinates.
Then we will have the following Lemma as Lemma 4.6, \cite{DL3}

\begin{lem}\label{lem-46}
There exist positive constants $c $ and  $\eta$,  depending only on the initial data and the constant
$\rho_0$ in
\eqref{co1}, for which
\begin{equation*}\label{derb1}
c \leq g_x(P)\leq c^{-1} \quad \text{and }\quad c \leq f_{xx}(P) \leq c^{-1}
\end{equation*}
at all points $P=(x,y,t)$ with  $f(P) >0$, $|P-P_0| \leq
\eta$ and $t \leq t_0$ under \eqref{nvec}.
\end{lem}

\subsection{Class of Linearized Equation}
In this subsection, we are going to show our transformed function $h$ from $g$ near the free boundary satisfies the same class of operators considered at \cite{DL3} so that  all the results in [DL3] can be applied to our equation by using the similar methods.

We will assume, throughout this section,  that at time
$t=0$ the function $g=(\beta \, f)^{\frac{1}{\beta}}$ satisfies the
hypotheses of Theorem \ref{thm-1.4} and that $g$ is smooth up to the
interface on $0 \leq t \leq T$, where $T>0$ is such that condition
\eqref{co1} holds. 

We will state the results of uniform $C^{1,\gamma}_s$-estimate in [DL3]. The readers can see detailed proofs with reference to [DL3].

Let
$P_0=(x_0,y_0,t_0)$ be  a point on the interface   curve $\Gamma(t_0)$  at time
$t=t_0$,  $0<t_0 \leq T$. We may assume, without loss of generality, that $\tau
\leq t_0 \leq T$, for some $\tau >0$. Indeed, the short time regularity result
in Theorem \ref{thm-1.3} shows that solutions are smooth up to the interface on $0 \leq t \leq
2\tau$, for some $\tau$ depending only on the initial data.
We may also assume,
by rotating the coordinates, that at the point $P_0$, condition
\eqref{nvec} holds.
By Lemma \ref{lem-46}, $g_x(P) >0$ for all points $P=(x,y,t)$ with $t \leq
t_0$, sufficiently close to $P_0$ and then from (\ref{eqn-hhhh})(see in [DH], Section II),
\begin{equation}\label{eqn-h}
 h_t  = -\frac{\left\{ z\, ( h_{zz} \, h_{yy} -h_{zy}^2) -  \, \theta(\al) \, h_z h_{yy}\right\}^{\alpha}}
{\left\{z^{2(\beta-1)} + h_z^2+ z^{2(\beta-1)}
 \, h_y^2\right\}^{\frac{4\alpha-1}{2}}}, \qquad z
>0.
\end{equation}
Set $K_h= z\, ( h_{zz} h_{yy} -h_{zy}^2) - \theta(\al)\,h_z\, h_{yy}$ and $J= \,z^{2(\beta-1)} + h_z^2\,+ \,z^{2(\beta-1)} \,h_y^2$.
If we linearize this equation around $h$, we obtain the equation
\begin{equation}\label{eqn-hy}
\begin{split}
\tilde h_t &= \frac{\al \,K_h^{\al-1} }{J^\frac{4\al-1}{2}}
\left\{  - z \, h_{yy} \tilde h_{zz} + 2 z\, h_{zy} \, \tilde
h_{zy} + ( \, \theta(\al) \,  h_z - z\, h_{zz})\, \tilde
h_{yy}\right\}+\, \frac{( \, 4\al-1)\, z^{2(\beta-1)}\,K_h^\al \,
h_y}{J^\frac{4\al+1}{2}}\, \tilde h_y
 \\
&\quad+\,\frac{( \, 4\al-1)\,K_h^\al \,h_z +\al K_h^{\al -1}\theta(\al)(h_z ^2 + z^{2(\beta -1)}(1+h_y ^2))h_{yy}}{J^\frac{4\al+1}{2}}  \,
\tilde h_z .
\end{split}
\end{equation}
Let us denote by  $\mathcal B_\eta$ the box
$$\mathcal B_\eta =
\{ \, 0\leq z \leq \eta^2
, \,\, |y-y_0| \leq \eta,
\,\, t_0-\eta^2 \leq t \leq t_0 \,\, \}
$$
around  the point $Q_0=(0,y_0,t_0)$.  We can obtain a-priori
bounds on the matrix
\begin{equation*}
A= (a_{ij})  = \left (
\begin{split}
- \,h_{yy} &\quad \sqrt z \,\, h_{zy}\\
\sqrt z\, \, h_{zy} &\quad (  \, \theta(\al) \, h_z-z\, h_{zz})
\end{split}
\right )
\end{equation*}
and the coefficient
\begin{equation*}
b = \frac{( \, 4\al-1)\,K_h^\al \,h_z +\al K_h^{\al -1}\theta(\al)(h_z ^2 + z^{2(\beta -1)}(1+h_y ^2))h_{yy}}{J^\frac{4\al+1}{2}} .
\end{equation*}
Then we obtain the following.

\begin{lem}\label{lem-51}
There exist positive constants
$\eta$, $\lambda $ and $\nu$, depending only on the initial data
and the constant $\rho_0$ in \eqref{co1} such that 
\begin{equation*}
 \lambda \, |\xi|^2 \leq a_{ij}\, \xi_i\, \xi_j  \leq
\lambda^{-1}
\, |\xi|^2, \qquad \forall \xi \neq 0
\end{equation*} and
\begin{equation*}
|b| \leq \lambda^{-1} \,\,\, \text{and}\,\,\, b \geq \nu >0 \,\,\, \text{on the box }\, {\mathcal B}_\eta.
\end{equation*}
\end{lem}
Notice that $b\geq\nu>0$ comes from the decay rates of second derivatives, Corollary \ref{cor-44}, and  Aronson-B\'{e}nilan type Estimate, Lemma \ref{lem-AB}. Similarly, we can get the bound of $\tilde A := (\tilde a_{ij})$ and $\tilde b_i$, $i=1,2$ to be
the coefficients
\begin{equation*}
\tilde a_{ij} = \frac{a_{ij}\al K_h ^{\al -1}}{(z^{2(\beta-1)}+ h_z^2 +
z^{2(\beta-1)}\, h_y^2)^{\frac {4\al-1}{2}}}
\end{equation*}
and
\begin{equation*}
 \tilde b_1 =b- \frac{\al\,\theta(\al) K_h^{\al-1} J h_{yy}}{J^\frac{4\al+1}{2}} \,
\qquad \mbox{and} \qquad \tilde b_2 =  \frac{(4\al-1)\,
z^{2(\beta-1)}\,K_h^\al \, h_y}{J^\frac{4\al+1}{2}}
\end{equation*}
of equation \eqref{eqn-hy}.

\begin{lem}\label{lem-52}
There exist constants
$\eta >0$, $\lambda  >0$ and $\nu >0$, depending only on the initial data
and the constant $\rho_0$ in \eqref{co1}, for which
\begin{equation*}
 \lambda \, |\xi|^2 \leq \tilde a_{ij}\, \xi_i\, \xi_j  \leq
\lambda^{-1}
\, |\xi|^2, \qquad \forall \xi \neq 0
\end{equation*} and
\begin{equation*}
|\tilde b_i| \leq \lambda^{-1}\,\,\, \text{and}\,\,\, \tilde b_1 \geq \nu >0 \,\,\, \text{on the box }\, {\mathcal B}_\eta.
\end{equation*}
\end{lem}

\medskip

\subsection{Regularity theory}
Recall the Definition \ref{def-sp}. Then Lemma \ref{lem-52} tells us the linearized equation \eqref{eqn-hy} is in the same class of operators considered  at Lemma 5.2 of \cite{DL3} and \cite{DL2}.

We are now in position to show the uniform H\"older  bounds of the
first order derivatives $h_t$, $h_y$ and $h_z$ of $h$ on $\mathcal B_\eta$.
In [DL3], the authors have obtained the
$C^{2,\gamma}_s$ regularity of $h$ on the box.

\begin{lem}\label{lem-53}
There exist numbers $\gamma$ and $\mu$ in $0 < \gamma,\mu <1$, and  positive
constants $\eta$ and $C$, depending only on the initial data and
$\rho_0$, such that
$$
\| h_y\|_{C^{2+\gamma}_s(\mathcal B_{\frac \eta2})} \leq C,\qquad \| h_t\|_{C^{2+\gamma}_s(\mathcal B_{\frac
\eta2})} \leq C,\qquad
\text{and} \qquad \| h_z\|_{C^\mu_s(\mathcal B_{\frac \eta2})} \leq C.
$$\end{lem}

Following the proof Theorem 6.2 of \cite{DL3}, we will have the following theorem.
\begin{thm}\label{thm-6.10}
With the same assumptions of Theorem \ref{thm-1.3} and condition \eqref{co1} which satisfies at $T < T_c$, there exist constants $0 < \alpha_0 <
1$, $C<\infty$ and $\eta >0$, depending only on the initial data
and $\rho_0$, for which $x=h(x,y,t)$ fulfills 
$$\|h\|_{C^{2+\gamma}_s(\cB_{\eta})} \leq C$$
on  $\cB_\eta= \{ \, 0\leq z \leq \eta^2 , \,\, |y-y_0| \leq \eta,
\,\, t_0-\eta^2 \leq t \leq t_0 \,\, \}$
for $P_0=(x_0,y_0,t_0)$ with $0 < \tau <  t_0 < T$, which is any free-boundary point holding condition \eqref{nvec}.
\end{thm}

\noindent{\bf Proof of Theorem \ref{thm-1.4}.}  By the short time existence
Theorem \ref{thm-1.3}, there exists a  maximal time $T >0$ for which $g$ is
smooth up to the interface on  $0 < t < T$. Assuming  that $T <
T_0$, we  will  show that at time $t=T$, the function  $g(\cdot,
T)$ is of class $C^{2+\gamma}_s$, up to the interface $z=0$, for
some $\gamma >0$, and  satisfies the non-degeneracy conditions
\eqref{eqn-condition}. Therefore, by Theorem [DH], there exists a number
$T' >0$ for which $g$ is of class  $C^{2+\gamma}_s$, for all $
\tau < T+T'$, and hence $C^\infty$ up to the interface, according
to Theorem 9.1 in \cite{DH}. This will contradict the fact that
$T$ is maximal, proving the Theorem.
From Lemma \ref{lem-31} and Corollary \ref{cor-44}, the  functions
$g(\cdot,t)$ satisfy conditions \eqref{eqn-condition}, for all $0 \leq t
< T$, with constant $c$ independent of $t$. Hence, it will be
enough to establish the uniform $C^{2+\gamma}_s$ regularity of
$g$, on $0 \leq t \leq T$, up to the interface, whose proof follows the same line of argument at \cite{DL3}.
\qed
\medskip

\noindent{\bf Acknowledgement.} Ki-Ahm Lee was supported by Basic Science Research
Program through the National Research Foundation of Korea(NRF)  grant funded by the Korea government(MEST)(2010-0001985).

}
\end{document}